\documentclass[12pt]{amsart}
\usepackage{etex}
\usepackage{amscd,amssymb,amsmath,mathrsfs,latexsym,upgreek,url}
\usepackage{mathtools}
\usepackage[all]{xy}
\usepackage{color,graphicx}
\usepackage{caption,subcaption}
\usepackage{hyperref}
\hypersetup{breaklinks=true}
\usepackage{todonotes}
\usepackage{enumitem}
\usepackage{float}
\usepackage{comment}
\usepackage[scale=0.75]{geometry}
\usepackage[utf8]{inputenc}
\usepackage[%
    style=alphabetic,sorting=nyt,
    maxnames=10,
    sortcites=true,doi=false,
    firstinits=true,hyperref,backend=bibtex]{biblatex}
    \renewbibmacro{in:}{%
      \ifentrytype{article}{}{\printtext{\bibstring{in}\intitlepunct}}}
\usepackage[babel]{csquotes}
\usepackage{guit}

\newcommand{\on}[1]{\operatorname{#1}}

\newcommand{\ov}[1]{{\overline{#1}}}
\newcommand{\ca}[1]{{\mathcal{#1}}}

\newcommand{\wt}[1]{{\widetilde{#1}}}

\newcommand{\GL}{\operatorname{GL}}

\newcommand{\cone}{\operatorname{cone}}
\newcommand{\rec}{\operatorname{rec}}

\renewcommand{\val}{\operatorname{val}}

\newcommand{\aff}{\textup{aff}}
\newcommand{\sph}{\textup{sph}}

\newcommand{\an}{\operatorname{an}}
\newcommand{\gr}{\operatorname{gr}}

\newcommand{\Spec}{\operatorname{Spec}}

\newcommand{\CH}{\operatorname{CH}}
\newcommand{\PP}{\operatorname{PP}}

\def \C{\mathbb{C}}

\def \N{\mathbb{N}}

\def \R{\mathbb{R}}

\def \Z{\mathbb{Z}}

\def\O {{\mathcal O}}

\def\E {{\mathcal E}}

\def\X {{\mathfrak X}}
\def\U {{\mathcal U}}
\def\B {{\mathfrak B}}
\def\m {{\mathfrak m}}
\def\k {{\bf k}}

\def \w{{\omega}}

\numberwithin{equation}{section}
\theoremstyle{definition}
\newtheorem{defn}[equation]{Definition}

\newtheorem{rem}[equation]{Remark}

\theoremstyle{plain}
\newtheorem{lem}[equation]{Lemma}

\newtheorem{prop}[equation]{Proposition}
\newtheorem{thm}[equation]{Theorem}
\newtheorem{cor}[equation]{Corollary}
\newtheorem{conj}[equation]{Conjecture}
\newtheorem{prop-def}[equation]{Proposition-Definition}

\begin{document}

\title[Equivariant Chern Classes of Toric Vector Bundles over a DVR]{Equivariant Chern Classes of Toric Vector Bundles over a DVR and Bruhat--Tits Buildings}
\author[Botero, Kaveh, Manon]{Ana Mar\'ia Botero and Kiumars Kaveh and Christopher Manon}

\thanks{A.\,Botero was funded by the Deutsche Forschungsgemeinschaft (DFG, German Research Foundation) – Project-ID 491392403 – TRR 358.  C. Manon was partially supported by Simons Collaboration Grant 587209 and National Science Foundation Grant DMS-2101911. K. Kaveh was partially supported by a National Science Foundation Grant (Grant ID: DMS-2101843) and a Simons Collaboration Grant for Mathematicians.}

\date{\today}

\maketitle


\begin{abstract}
We define equivariant Chern classes of a toric vector bundle over a proper toric scheme over a DVR. We provide a combinatorial description of them in terms of piecewise polynomial functions on the polyhedral complex associated to the toric scheme, which factorize through to an extended Bruhat--Tits building. We further motivate this definition from an arithmetic perspective, connecting to the non-Archimedean Arakelov theory of toric varieties.
\end{abstract}

\tableofcontents

\section{Introduction}
Combinatorial descriptions of equivariant Chern classes of toric vector bundles over a field have appeared in e.\,g.\,\cite[Prop. 3.1]{Payne-chern}, or more recently in \cite[Corollary 3.5]{KM}. In the latter, the main gadget is the classification of toric vector bundles in terms of piecewise linear maps to the extended (spherical) Tits building associated to the general linear group. We recall the result (see also Theorem \ref{th-equiv-Chern-tvb-field} in the text).
\begin{thm}
Let $\E$ be a rank $r$ toric vector bundle over a toric variety $X_{\Sigma}$ with $\Phi_{\E} \colon |\Sigma| \to \widetilde{\B}_{\sph}\left(\on{GL}(r)\right)$ its corresponding piecewise linear map. Then for any $1 \leq i \leq r$, the $i$-th equivariant Chern class $c_T^i(\E)$ is represented by the piecewise polynomial function $\epsilon_i \circ \Phi_{\E}$, where $\epsilon_i: \wt{\B}_\sph(\GL(r)) \to \R$ denotes the $i$-th elementeary symmetric function on the extended spherical Tits building. 
\end{thm}
Note that the equivaraint Chern class $c_T^i(\E)$ is a homogeneous degree-$i$ element in the equivariant Chow cohomology ring of the toric variety $\CH^*_T(X_{\Sigma})$. The latter is a graded algebra isomorphic to the graded algebra of piecewise polynomial functions on the fan $\Sigma$ (see Theorem~\ref{th:payne-chow}).

The main goal of this article is a similar description of equivariant Chern classes of toric vector bundles over a discrete valuation ring (DVR). The rough idea is as follows. Toric vector bundles over a DVR are classified by piecewise affine maps from the (support of the) polyhedral complex associated to the toric scheme into an (extended) affine Bruhat--Tits building (this is shown in \cite{KMT} and recalled in Theorem \ref{th:vb-dvr}). In this article, we define the equivariant Chern classes of a toric vector bundle $\E$ over a toric scheme over a DVR as the equivariant Chern classes of the restriction of $\E$ to the special fiber (see Definition \ref{def:eq-chern}). The equivariant Chern classes of $\E$ are thus elements in the equivariant Chow cohomology ring of the special fiber, which again has a description in terms of piecewise polynomial functions (see Section \ref{sec:equiv-chow}). Similar to the case over a field, we describe representatives in terms of the piecewise affine map and the elementary symmetric functions on the (extended) affine Bruhat--Tits building (see Theorem \ref{th:intro} for the precise statement). 

We now present a more detailed version of our results. 

Let $K$ be a discretely valued field with valuation $\val : K \to \ov{\Z} = \Z\cup \{\infty\}$. We denote by $\O$ the corresponding
valuation ring. The scheme $S=\Spec(\O)$ has two points: the
generic point $\eta$ corresponding to the prime ideal $\{0\}$ and the special point $s$ corresponding
to the maximal ideal $\m$. For a scheme $\X$ over $S$, we denote by $\X_{\eta}$ the
generic fiber, and by $\X_s$ the special fiber. 

Let $T$ be a split torus over $S$ of relative dimension $n$. We denote the base change to $\operatorname{Spec}(K)$ and to $\operatorname{Spec}(\k)$ by $T_{\eta}$ and $T_{s}$, respectively. A toric scheme over $S$ of relative dimension $n$ is a normal integral
separated scheme of finite type $\X$ over $S$ equipped with a dense open embedding $T_{\eta} \hookrightarrow \X_{\eta}$ and an action of $T$ on $\X$ over $S$ which extends the translation action
of $T_{\eta}$ on itself. While toric varieties (over a field) are classified by fans in $N_{\R}$ (where $N$ is the cocharacter lattice of the torus), toric
schemes (over $S$) are classified by fans in $N_{\R} \times \R_{\geq0}$ (we refer to Section \ref{sec:toric-schemes} for details regarding this classification). For a fan $\Sigma$ in
$N_{\R} \times \R_{\geq 0}$ we denote the corresponding toric scheme by $\X_{\Sigma}$. By intersecting cones in $\Sigma$ with
$N_{\R} \times \{1\}$ one obtains a rational polyhedral complex $\Sigma_1$ in the affine space $N_{\R} \times \{1\}$. The generic
fiber $\X_{\eta}$ is a usual toric variety over $K$ with corresponding fan $\Sigma_0$ obtained by intersecting
the cones in $\Sigma$ with $N_{\R} \times \{0\}$. The (reduced) special fiber $\X_s$ is a union of toric varieties indexed by the polyhedra in $\Sigma_1$, glued in a way that is compatible with the rational polyhedral complex.  Moreover, a cone $\sigma$ in $\Sigma_0$ induces a horizontal cycle $V(\sigma)$ in $\X_{\Sigma}$, in the sense that the structure morphism $V(\sigma) \to S$ is dominant,
of relative dimension $n - \on{dim}(\sigma)$, while a polyhedron $\Delta \in \Sigma_1$ gives a vertical cycle $V(\Delta)$ in $\X_{\Sigma}$ in the sense
that its image by the structure morphism is the closed point $s$.
The cycle $V(\Delta)$ has itself the structure of a toric variety of dimension $\on{codim}(\Delta)$. 

Let $\E$ be a toric vector bundle over a toric scheme $\X_{\Sigma}$, that is a $T$-linearized vector bundle $\E$ over $\X_{\Sigma}$. As evidenced in \cite{KM} in the field case, and in \cite{KMT} in the DVR case, the correct objects to classify such vector bundles are spherical Tits and affine Bruhat--Tits buildings, respectively. Roughly speaking, a building is an (infinite) simplicial complex together with distinguished subcomplexes called apartments, which satisfy a list of axioms. There are two important types of buildings: affine and spherical. In an affine building, the apartments correspond to triangulations of an affine space, whereas in the spherical building, apartments correspond to triangulations of spheres. 

Let $E \simeq K^r$ be an $r$-dimensional $K$-vector space. We will only consider the spherical Tits building $\B_{\sph}(E)$ and the affine Bruhat--Tits building $\B_{\aff}(E)$ associated to the general linear group $\GL(E)$. 
We describe the realization spaces of both buildings in terms of the space $\tilde{\B}(E)$ of all \emph{level $m$ additive norms} (see Section \ref{sec:buildings}). We call the space of all level $m$ additive norms the \emph{total extended building of $\GL(E)$}. The space of level $0$ additive norms consists of valuation on $E$ extending the trivial valuation on $K$. After quotienting this space by a natural equivalence relation, we get $\B_\sph(E)$, (the realization space of) the Tits building of $\GL(E)$ (see Remark \ref{rem-realization-space-bldg}(i)). We denote the space of level $0$ additive norms by $\tilde{\B}_\sph(E)$ and call it the \emph{extended Tits building of $\GL(E)$}. Similarly, the space of additive norms of level $1$ consists of valuations on $E$ extending $\val$ on $K$. After quotienting this space by a natural equivalence relation, we get $\B_\aff(E)$, (the realization space of) the Bruhat-Tits building of $\GL(E)$ (see Remark \ref{rem-realization-space-bldg}(ii)). We denote the space of level $1$ additive norms by $\tilde{\B}_\aff(E)$ and call it the \emph{extended Bruhat--Tits building of $\GL(E)$}.

As we already mentioned, it is shown in \cite{KM} that toric vector bundles over a toric variety over a field are classified by piecewise linear maps from the (support of the) fan of the toric variety to the extended Tits building, whereas in \cite{KMT} it is shown that toric vector bundles over a toric scheme over $S$ are classified by piecewise affine maps from the (support of the) polyhedral complex associated to the toric scheme to the extended Bruhat--Tits building (this classification result is recalled in Section~\ref{sec:class}). 

Given a proper regular toric scheme $\X_{\Sigma}$ and a toric vector bundle $\E$ over $\X_{\Sigma}$, denote by $\Phi_{\E, 1}\colon |\Sigma_1| \to \B_{\aff}(E)$ the associated piecewise affine map. 

We define the equivariant $i$-th Chern class $c_i^T(\E)$ of $\E$ by its restriction to the special fiber, i.\,e.
\[
c_i^T(\E) \coloneqq c_i^T(\iota^*\E) \in \CH^i_{T_s}(\X_s),
\]
where $\iota \colon \X_s \hookrightarrow \X_{\Sigma}$ denotes the canonical inclusion morphism. 

An element in $\CH^i_{T_s}(\X_s)$ is represented by a tuple of piecewise polynomial functions $f =\left(f_\nu\right)_{\nu}$, where the index set consists of all vertices of $\Sigma_1$, and the $f_\nu$ are piecewise polynomial functions of degree $i$ on the fan $\Sigma_1(\nu)$, the star of $\Sigma_1$ at $\nu$, subject to some compatibility conditions (see Section \ref{sec:equiv-chow} for details, in particular Definition \ref{def:pp-poly}).

On the other hand, consider $\Lambda \in \wt{\B}_{\aff}(E)$, the image of a vertex $\nu$ under the piecewise affine map $\Phi_{\E, 1}$. This is a vertex of of the extended affine building, whose link can be identified with an extended spherical Tits building $\wt{\B}_{\sph}(E_{\Lambda})$ (Proposition \ref{prop-link-vertex}). On $\wt{\B}_{\sph}(E_{\Lambda})$ one has the $i$-th elementary symmetric function $\epsilon_i$ (see Section \ref{sec:class}). 

The following are the main results of this paper (Theorems \ref{th:link} and \ref{th:equi} in the text). We write $\X_{s,\nu}$ for the irreducible component of the special fiber $\X_s$ corresponding to a vertex $\nu \in \Sigma_1$.

\begin{thm}[Piecewise linear map corresponding to restriction of $\E$ to $\X_{s,\nu}$]  \label{th-intro1}
Consider the piecewise linear map $\Phi_\nu: |\Sigma_1(\nu)| \to \wt{\B}_\sph(E_\Lambda)$ obtained by restricting the piecewise affine map $\Phi_1=\Phi_{\E, 1}$ to a small neighborhood of $\nu$ in $\Sigma_1(\nu)$ and by looking at a small neighborhood (link) of $\Lambda = \Phi_1(\nu)$ in $\wt{\B}_\aff(E)$ which we can naturally identify with a small neighborhood of the origin in $\wt{\B}_{\sph}(E_\Lambda)$ (Proposition \ref{prop-link-vertex}). Then $\Phi_\nu$ is the piecewise linear map corresponding to the restriction of $\E$ to the $T_{s}$-toric variety $\X_{s,\nu}$.
\end{thm}

\begin{thm}\label{th:intro}
    Let $\E$ be a toric vector bundle on a toric scheme $\X_\Sigma$. Let $c_i^T(\E)$ denote the $i$-th equivariant Chern class of $\E$ represented by a piecewise polynomial function $(f_\nu)_{\nu}$, where $\nu$ runs over the vertices in the polyhedral complex $\Sigma_1$. Then the piecewise polynomial function $f_\nu$ on the fan $\Sigma_1(\nu)$, is given by $\epsilon_i \circ \Phi_\nu$. 
\end{thm}
One can also give a description of the equivariant Chern classes of the toric vector bundle restricted to the generic fiber (see Corollary \ref{cor:recovering}). 

In the last section \ref{sec:arakelov}, we motivate our definition of equivariant Chern classes from an arithmetic perspective and establish connections with the non-Archimedean Arakelov theory of toric varieties. More specifically, the non-Archimedean analogue of a smooth toric metric on a toric vector bundle $\E$ on a toric variety defined over a non-Archimedean field $K$ is the choice of a model $\X$ and an extension $\wt{\E}$ of $\E$ to $\X$. Metrics arising in this way are called (toric) \emph{model} metrics. 

Hence, equivariant Chern classes of toric vector bundles over a toric scheme over a DVR are the non-Archimedean analogues of the (complex differential) Chern forms of hermitian vector bundles. One can then consider more general continuous toric metrics by considering uniform limits of toric model metrics. Generalizing the case of line bundles studied in \cite{BPS}, we expect, under some positivity assumption, a classification of continuous toric metrics on toric vector bundles in terms of piecewise affine maps to the (extended) Bruhat--Tits buildings satisfying some asymptotic conditions. The following is Conjecture \ref{conj:class-metrics} in the text. 
\begin{conj}
 Let $\E$ be a rank $r$-toric vector bundle over a complete toric variety $X$ over the discretely valued field $K$. Write $\E^{\on{an}}$ for the analytifiaction of $\E$ in the sense of Berkovich. Then there is a bijection between the space of semipositive toric metrics on $\E^{\on{an}}$ and the space of ``concave'' piecewise affine functions $N_{\R}  \to \wt{\B}_{\aff}(E)$ satisfying an asymptotic growth condition with respect to the piecewise linear map $\Phi_{\E} \colon N_{\R} \to \wt{\B}_{\sph}(E)$ corresponding to $\E$.
\end{conj}

\section{Notation}
\begin{itemize}
\item[-]$(K, \val) =$ field with discrete valuation $\val$ with values in $\ov{\Z} = \Z \cup \{\infty\}$
\item[-] $\O =$ the valuation ring of $(K, \val)$
\item[-] $\m =$ the maximal ideal of $\O$
\item[-] $\varpi =$ a uniformizer of $\O$
\item[-] $\k =$ the residue field of $\O$
\item[-] $S = \Spec(\O)=$ Spectrum of the valuation ring
\item[-]$\eta$ and $s$ for generic and special points of $S$, respectively
\item[-] $\X_{\eta}$ and $\X_s$ for the generic and special fiber of a scheme $\X$ over $S$ respectively

\item[-]$\Sigma =$ fan in $N_\R \times \R_{\geq 0}$ (not necessarily complete)
\item[-] $\Sigma_1 =$ polyhedral complex in $N_\R \times \{1\}$ obtained by intersecting cones in $\Sigma$ with $N_\R \times \{1\}$
\item[-] $\Sigma_0 =$ fan in $N_\R \times \{0\}$ obtained by taking cones in $\Sigma$ that lie in $N_\R \times \{0\}$ (recall that if $\Sigma$ is complete then $\Sigma_1$ is a complete polyhedral complex and $\Sigma_0$ is its recession fan).
\item[-]$\Sigma_i(\ell) =$ set of $\ell$-dimensional polyhedra in $\Sigma_i$ for $i = 0,1$
\item[-] $\Sigma_i(\gamma)=$ star of $\Sigma_i$ at $\gamma \in \Sigma_i$, for $i = 0,1$  

\item[-] $\U_{\sigma} =$ affine toric scheme associated to a cone $\sigma \in \Sigma$
\item[-] $\X_{\Sigma} =$ toric scheme associated to the fan $\Sigma$
\item[-] $x_0 =$ the distinguished point in the open orbit $\U_0$ corresponding to $1 \in T_{\eta}$

\item[-] $\E$= toric vector bundle over $\X_{\Sigma}$ of rank $r$
\item[-] $E = \E_{x_0} \cong K^r =$ fiber of $\E$ over the distinguished $K$-point $x_0$

\item[-] $\wt{\B}_{\aff}(E) =$ space of all additive norms on $E$, we also refer to it as the extended Bruhat--Tits building of $\GL(E)$
\item[-] $\wt{\B}_{\sph}(E) =$ space of all valuations on $E$, we also refer to it as the extended Tits building of $\GL(E)$ 
\item[-] $\wt{\B}_m(E)=$ level $m$ extended building, it consists of functions $\w: E \to \ov{\R}$ which satisfy:
\begin{enumerate}
\item[(1)] non-Archimedean inequality
\item[(2)] $\w(\lambda e) = m\lambda\val(e) + \w(e)$, $\forall e \in E$, $\forall \lambda \in K.$
\item[(3)] $\w(e) = \infty$ iff $e = 0.$

\end{enumerate}
Then, $\wt{\B}_0 = \wt{\B}_{\sph}(E)$ and $\wt{\B}_1(E) = \wt{\B}_{\aff}(E)$
\item[-]$\wt{\B}(E) = \bigcup_{m \geq 0} \wt{\B}_m(E) =$ total extended building

\item[-] $\Phi: |\Sigma| \to \wt{\B}(E) =$ a graded piecewise LINEAR map from $\Sigma$ to the total extended building.

\item[-] $\Phi_1: |\Sigma_1| \to \wt{\B}_\aff(E) =$ the piecewise AFFINE map from the polyhedral complex $\Sigma_1$ to the extended affine building (obtained from $\Phi$).

\item[-] $\Phi_0: |\Sigma_0| \to \wt{\B}_\sph(E) =$ the piecewise linear map from the fan $\Sigma_0$ to the extended spherical building (obtained from $\Phi$).

\item[-] $\CH^k_G(X)=$ equivariant operational $k$-th Chow cohomology ring of a scheme $X$ with a $G$-action.
\end{itemize}
\section{Toric schemes over a DVR}\label{sec:toric-schemes}
Let $\O$ be a discrete valuation ring with fraction field $K$. We denote by $\varpi$ a uniformizer for $\O$, i.~e. a generator of the principal maximal ideal $\m$ of $\O$, and we let $\k = \ca O/\m$ be the residue field. We let $s$ and $\eta$ denote the special and generic points of $S=\Spec(\ca O)$, corresponding to the maximal ideal $\m$ and the prime ideal $(0)$, respectively. For a scheme $\X$ over $S$ we denote by $\X_s = \X \times_{S}\Spec(\k)$ and by $\X_{\eta} = \X \times_S\Spec(K)$ the special and the generic fiber, respectively.

We let $M$ and $N$ be dual free abelian groups of rank $n$, and we set $\widetilde{N} = N \times \Z$ and $\widetilde{M} = M \times \Z$. Let $T$ denoted a split torus over $S$. We denote by $T_{\eta}$ and by $T_s$ the base change to $\operatorname{Spec}(K)$ and to $\operatorname{Spec}(\k)$, respectively. We suppose that $T_{\eta} \simeq (K^{\times})^n$ has character lattice $M$.

A \emph{toric scheme} $\X$ over $S$ of relative dimension $n$ is a normal integral separated scheme of finite type, equipped with a dense open embedding $T_{\eta} \hookrightarrow \mathcal{X}_{\eta}$ and an $S$-action of $T$ over $\mathcal{X}$ that extends the action of $T_{\eta}$ on itself by translations. The image of the identity point of $T_{\eta}$ in $\X$ is denoted by $x_0$. 

Note that if $\mathcal{X}$ is a toric scheme over $S$, then $\mathcal{X}_{\eta}$ is a toric variety over $K$ with torus $T_{\eta}$.

We now recall briefly the combinatorial classification of toric schemes over $S$ in terms of complete, strongly convex, rational polyhedral complexes (we refer to \cite[Section~3.5]{BPS} for details on the classification and to \cite[Section~2.1]{BPS} for definitions regarding strongly convex, rational polyhedral complexes and fans).

From now on, when we say cone/polyedron/fan/polyhedral complex, we mean always a strongly convex, rational cone/polyedron/fan/polyhedral complex. And for a cone we also always mean a polyhedral cone. 

Let $\Sigma$ be a fan in $N_{\R} \times \R_{\geq 0}$. We denote by $\Sigma_1$ the polyhedral complex in $N_{\R} \times \{1\}$ obtained by intersecting all the cones in $\Sigma$ by $N_{\R} \times \{1\}$. 

Similarly, we denote by $\Sigma_0$ the fan in $N_{\R} \times \{0\}$ obtained by intersecting all the cones in $\Sigma$ by $N_{\R} \times \{0\}$.

\begin{rem}
    Under some conditions, which hold for example when $\Sigma$ is complete in $N_{\R} \times \R_{\geq 0}$, the fan $\Sigma_0$ coincides with the \emph{recession fan of} $\Sigma_1$
\[
\rec(\Sigma_1) = \left\{\rec(\Delta) \times \{0\} \; | \; \Delta \in \Sigma_1\right\}.
\]
In this case, let $c(\Sigma_1)$ be the polyhedral complex in $N_{\R} \times \R_{\geq 0}$ consisting of the cones over the polyhedra in $\Sigma_1$ together with its recession cones, i.~e.
\[
c(\Sigma_1) \coloneqq \left\{\cone(\Delta) \; | \; \Delta \in \Sigma_1\right\} \cup \rec(\Sigma_1).
\]
This is a fan in $N_{\R} \times \R_{\geq 0}$ and the above construction is bijective, i.\,e. we get $\Sigma = c(\Sigma_1)$. Strong convexity and rationality are also preserved. In particular, we obtain a bijective correspondence between \emph{complete} polyhedral complexes in $N_{\R}\times \{1\}$ and \emph{complete} fans in $N_{\R} \times \R_{\geq 0}$. Note that in general this is not true (see \cite{BS} for an example). 
\end{rem}

We now explain how to construct a toric scheme from a fan $\Sigma$ in $N_{\R} \times \R_{\geq 0}$. This is done as in the case of toric varieties over a field, by first associating an affine toric scheme $\X_{\sigma}$ to a cone $\sigma \in \Sigma$ and then applying a gluing construction. 

Let $\sigma \subset N_{\R} \times \R_{\geq 0}$ be a cone in $\Sigma$. We denote by $\sigma^{\vee} \subset \wt{M}_{\R}$ its dual cone. Let $R_{\sigma}$ be the subring of the ring of Laurent polynomials $K[T_{\eta}]$ defined by 
\[
R_{\sigma} = \ca O\left[\chi^u \varpi^k \; | \; (u,k) \in \sigma^{\vee}\cap \widetilde{M} \right],
\]
and let $\U_{\sigma} \coloneqq \Spec(R_{\sigma})$. When $\sigma$ is the cone over a polyhedron $\Delta$ in $N_{\R} \times \{1\}$, we also denote $R_{\sigma}$, $\U_\sigma$ by $R_{\Delta}$, $\U_\Delta$ respectively. 

As usual, one constructs the toric scheme $\X_{\Sigma}$ by gluing the affine schemes $\R_{\sigma}$ for $\sigma \in \Sigma$. If $\Sigma = c(\Sigma_1)$ for some complete polyhedral complex $\Sigma_1$ in $N_{\R}\times \{1\}$, then we also denote the toric scheme $\X_{\Sigma}$ by $\X_{\Sigma_1}$. 

We now gather the main properties of toric schemes. We refer to \cite[Chap IV]{KKMD} and to \cite[Section~3.5]{BPS} for proofs and further details.

\begin{enumerate}
\item The association $\sigma \mapsto \U_{\sigma}$ gives an equivalence of categories between the category of cones in $N_{\R}\times \R_{\geq 0}$ and the category of affine toric schemes over $S$. 
\item The association $\Sigma \mapsto \X_{\Sigma}$ gives an equivalence of categories between fans in $N_{\R} \times \R_{\geq 0}$ and toric schemes over $S$. Moreover, $\X_{\Sigma}$ is proper and/or regular if and only if $\Sigma$ is complete and/or regular. $\Sigma$ is said to be regular if any cone is generated by a subset of a basis of $N \times \Z$. 
\item There are two different types of cones in $\Sigma$: the ones that are contained in $N_{\R} \times \{0\}$ and the ones that are not. The former are in bijective correspondence with the orbits of $T_{\eta}$ in $\X_{\Sigma,\eta}$. In particular, the generic fiber $\X_{\Sigma,\eta}$ is isomorphic to $X_{\Sigma_0}$, the toric variety over $K$ corresponding to $\Sigma_0$. The cones that are not contained in $N_{\R} \times \{0\}$ are of the form $\cone(\Delta)$ for some $\Delta \in \Sigma_1$ and are in bijective correspondence with the orbits of $T_{s}$ in the special fiber $\X_{\Sigma,s}$. Note that the special fiber $\X_{\Sigma,s}$ has an induced action by $T_{s}$ but, in general, it is not a toric variety over $\k$ (it is not necessarily irreducible nor reduced). It is reduced if and only if the vertices of all $\Delta \in \Sigma_1$ are in $N \times \{1\}$. In general, it is a union of toric varieties (with multiplicities) corresponding to polyhedra in $\Sigma_1$, which intersect according to the combinatorics of $\Sigma_1$. The correspondence is order reversing. In particular, vertices in $\Sigma_1$ correspond to irreducible components of the special fiber and two such components intersect if and only if there is an polyhedron in $\Sigma_1$ containing both vertices.  

\end{enumerate}

\begin{defn}\label{def:vertical-cycles}
Let $\Sigma$, $\Sigma_1$ and $\Sigma_0$ be as above. 
    \begin{enumerate}
    \item  For $\sigma \in \Sigma_0$ we denote by $O_{\sigma}\subset \X_{\Sigma,\eta}$ the corresponding $T_{\eta}$-orbit and by $V(\sigma)$ its Zariski closure in $\X_{\Sigma}$. By \cite[Proposition~3.5.7]{BPS}, this is again a toric scheme over $S$ of relative dimension $\on{codim}(\sigma)$. 
    \item For $\Delta \in \Sigma_1$ we denote by $O_{\Delta}\subset \X_{\Sigma,s}$ the corresponding $T_{s}$-orbit and by $V(\Delta)$ its Zariski closure in $\X_{\Sigma}$, which is contained in the special fiber $\X_{\Sigma,s}$. In the reduced case, $V(\Delta)$ is a toric variety over $\k$ of dimension $\on{codim}(\Delta)$ with corresponding fan $\Sigma_1(\Delta)$, the \emph{star} of $\Sigma_1$ at $\Delta$ (see \cite[Proposition 3.5.8]{BPS}).  
    \end{enumerate}
\end{defn}

\section{Equivariant Chow groups of special fibers of toric schemes}\label{sec:equiv-chow}
Given an algebraic group $G$ and a scheme $X$ with a $G$-action we consider the equivariant operational Chow cohomology groups $\CH^*_G(X)$ from \cite{EG}. We will be only interested in the case where $G$ is a torus and where $X$ is either a toric variety over a field or the special fiber of a toric scheme over $S$. In the former case, we have the following description given by Payne of the equivariant operational Chow cohomology ring (see \cite{Payne-equi}).
\begin{thm}\label{th:payne-chow}
    Let $X_{\Omega}$ be a complete toric variety over a field $F$, with torus $T_F$, corresponding to a complete fan $\Omega$ in $N_{\R}$. Let $\PP^*(\Omega)$ be the graded algebra of piecewise polynomial functions on $N_{\R}$ with respect to $\Omega$. There is a natural isomorphism of graded rings 
    \[
    \CH^*_{T_F}(X_{\Omega}) \simeq \PP^*(\Omega).
    \]
\end{thm}
We note that the above theorem for the usual equivariant cohomology ring of smooth complete toric varieties goes back to Brion (see \cite[Proposition 2.2]{Brion-PP}).

Now, let $\Sigma, \Sigma_1$ and $\Sigma_0$ as in the previous section and assume that $\Sigma$ is complete and regular. We also assume that the vertices of all $\Delta \in \Sigma_1$ are in the lattice $N \times \{1\}$. For an integer $\ell \in \Z_{\geq 0}$ we denote by $\Sigma_1(\ell)$ the set of $\ell$-dimensional polyhedra in $\Sigma_1$. Let $\X = \X_{\Sigma_1}$ be the corresponding proper, regular toric scheme. We now provide a combinatorial description of the equivariant Chow cohomology ring $\CH^*_{T_s}(\X_s)$ of the special fiber. 

Let $t = |\Sigma_1(0)|$, i.~e. the number of vertices of $\Sigma_1$. Recall that this equals the number of irreducible components of $\X_s$. 

For any $i \in \{1, \dotsc, t\}$ let $\X_s^{[i]}$ be the disjoint union of $i$-fold intersections of irreducible components of $\X_s$. Set also $\X_s^{[0]}=\X$ and $\X_s^{[t]}= \emptyset$ if $t >n$. 

We will only consider the cases $i = 0,1,2$. Using the correspondence between vertical cycles and polyhedra in $\Sigma_1$ we can write
\[
\X_s^{[1]} = \bigsqcup_{\nu \in \Sigma_1(0)}V(\nu) \quad \text{ and }\quad \X_s^{[2]} = \bigsqcup_{\stackrel{\gamma \in \Sigma_1(1)}{\gamma \text{ bounded}}}V(\gamma),
\]
where $V(\Delta)$ denotes the vertical cycle corresponding to $\Delta \in \Sigma_1$. As was recalled in Section~\ref{sec:toric-schemes}, this is a complete toric variety over $\k$ with corresponding fan $\Sigma_1(\Delta)$. 

By Payne's result we can identify $\CH_{T_s}^*\left(V(\Delta)\right)$ with the graded ring of piecewise polynomial functions $\PP^*(\Sigma_1(\Delta))$. On the other hand, by \cite[Proposition 4.2]{botero-equiv}, we have that for any $p \in \Z_{\geq 0}$ the equivariant operational Chow cohomology of the special fiber $\X_s$ fits into an exact sequence

\[
0 \to \CH^p_{T_{s}}\left(\X_s\right) \to \CH^p_{T_{s}}\left(\X_s^{[1]}\right) \xrightarrow{\rho} \CH^p_{T_{s}}\left(\X_s^{[2]}\right). 
\]
The description of the map $\rho$ is given explicilty in \cite[Section 4]{botero-equiv}. 

In order to describe the equivariant operational Chow ring of $\X_s$ we make the following definition. 
\begin{defn}\label{def:pp-poly}
    Let 
    \[
    \PP^*(\Sigma_1) \coloneqq \bigoplus_{p \geq 0}\PP^p(\Sigma_1)
    \]
    be the graded ring of piecewise polynomial functions on $\Sigma_1$ given as follows: for any $p \in \Z_{\geq 0}$, the space $\PP^p(\Sigma_1)$ consists of tuples of piecewise polynomial functions $\left(f_{\nu}\right)_{\nu \in \Sigma_1(0)}$ subject to the following conditions:
\begin{enumerate}
\item[(i)] $f_{\nu} \in \PP^p(\Sigma_1(\nu))$ for any vertex $\nu \in \Sigma_1(0)$, i.~e.~$f_{\nu}$ is a piecewise polynomial function of degree $p$ with respect to the fan $\Sigma_1(\nu)$. For any (full-dimensional) polyhedron $\Delta \in \Sigma_1$ containing a vertex $\nu \in \Sigma_1(0)$, denote by $f_{\nu,\Delta}$ the polynomial in $N_{\R}$ which represents the restriction of $f_{\nu}$ to $\Delta \in \Sigma_1(\nu)$.
\item[(ii)] For any (full-dimensional) polyhedron $\Delta \in \Sigma_1$ and any two vertices $\nu$ and $\nu'$ in $\Delta$, we have that $f_{\nu,\Delta}=f_{\nu',\Delta}$ as polynomials in $N_{\R}$.
\end{enumerate}
\end{defn}

\begin{rem} Note that the cone $C_{\nu}$ corresponding to $\Delta$ in $\Sigma_1(\nu)$ and the cone $C_{\nu'}$ corresponding to $\Delta$ in $\Sigma_1(\nu')$ are different cones. What we require is that the piecewise polynomial function $f_{\nu}$ restricted to $C_{\nu}$ and the piecewise polynomial function $f_{\nu'}$ restricted to $C_{\nu'}$ are represented by the same polynomial. 
\end{rem}
The following is \cite[Theorem 4.6]{botero-equiv}.
    \begin{thm}\label{th:cohomology-special}
    There is a natural isomorphism of graded rings 
    \[
    \CH_{T_s}^*(\X_s) \simeq \PP^*(\Sigma_1).
    \]
\end{thm}
\begin{rem}
        One can also provide a combinatorial description of the relative equivariant Chow cohomology $\CH^*_T(\X/S)$ of the toric scheme $\X$ over $S$ (see \cite[Section 3]{botero-equiv}), however we will not make use of this description here.
    \end{rem}

\section{Tits and Bruhat--Tits buildings of $\GL(r)$}\label{sec:buildings}
As before, let $\val: K \to \overline{\Z}$ be a discretely valued field with valuation ring $\mathcal{O}$ and residue field $\k$. Let $E$ be an $r$-dimensional vector space over $K$.

We consider two kinds of valuations on $E$: the ones that extend the trivial valuation on $K$, and the ones that extend $\val$. We incorporate both in the following definition.

\begin{defn}[Level $m$ additive norm]  \label{def-level-m-add-norm}
For $m \geq 0$, we call a function $\w: E \to \overline{\R} = \R \cup \{\infty\}$ a {\it level $m$ additive norm} if the following hold:
\begin{itemize}
	\item[(1)] For all $e \in E$ and $0 \neq \lambda \in K$ we have $\w(\lambda e) = m \val(\lambda) + \w(e)$. 
	\item[(2)] For all $e_1, e_2 \in E$, the non-Archimedean inequality $\w(e_1+e_2) \geq \min\{\w(e_1), \w(e_2)\}$ holds.
	\item[(3)] $\w(e) = \infty$ if and only if $e=0$. 
\end{itemize}
We refer to an additive norm of level $1$ simply as an \emph{additive norm},  and we call an additive norm of level $0$ a \emph{valuation}.
\end{defn}


\begin{defn}[Extended buildings] The following objects play a central role in the paper:
\begin{itemize}
\item[(a)] We denote the set of all additive norms on $E$ by $\wt{\B}_\aff(E)$ and call it the \emph{extended Bruhat--Tits building of $E$}. 
\item[(b)] We denote the set of all valuations on $E$ by $\wt{\B}_\sph(E)$ and call it the \emph{extended Tits building of $E$} (or the \emph{cone over the Tits building of $E$}).
\item[(c)] For $m \geq 0$, we denote the set of all level $m$ additive norms by $\wt{\B}_m(E)$. Clearly $\wt{\B}_1(E) = \wt{\B}_\aff(E)$ and $\wt{\B}_0(E) = \wt{\B}_\sph(E)$. We let $$\wt{\B}(E) = \bigcup_{m \geq 0} \wt{\B}_m(E),$$
and call it the \emph{total extended building} of $E$.
\end{itemize}
\end{defn}

\begin{rem}
We note that $\wt{\B}(E)$ comes with a natural action of the multiplicative semigroup $\R_{\geq 0}$: if $\w: E \to \overline{\R}$ is a level $m$ additive norm and $k \geq 0$ then the function $k \w$ is a level $km$ additive norm. 
\end{rem}

A \emph{frame} is a direct sum decomposition of $E = \sum_{i=1}^r L_i$ into one-dimensional subspaces $L_i$. In other words, frames correspond to vector space bases up to multiplication of basis elements by nonzero scalars. 

We say that a level $m$ additive norm $\w: E \to \overline{\R}$ is \emph{adapted} to a frame $L=\{L_1, \ldots, L_r\}$ for $E$ if the following holds. For any $e = \sum_i e_i$, $e_i \in L_i$, we have: 
$$\w(e) =  \min\{ \w(e_i) \mid i=1, \ldots, r\}.$$
In other words, if $B = \{b_1, \ldots, b_r\}$ is a basis with $b_i \in L_i$, then for any $e = \sum_i \lambda_i b_i$ we have: $$\w(e) = \min\{ m\val(\lambda_i) + \w(b_i) \mid i=1, \ldots, r\}.$$ That is, the additive norm $\w$ is determined by its values on the basis $B$. 

\begin{defn}[Extended apartment] \label{def-ext-apt}
Let $L = \{L_1, \ldots, L_r\}$ be a frame for $E$. We denote the collection of all level $m$ additive norms that are adapted to $L$ by $\wt{A}_m(L)$ and call it the extended level $m$ apartment of $L$. We put $\wt{A}(L) = \bigcup_{m \geq 0} \wt{A}_m(L)$. We also set 
$\wt{A}_\sph(L) = \wt{A}_0(L)$ and $\wt{A}_\aff(L) = \wt{A}_1(L)$.
If $B$ is a basis for $E$ spanning a frame $L$, we also denote the extended apartments corresponding to $L$ by $\wt{A}(B)$, $\wt{A}_\sph(B)$ and $\wt{A}_\aff(B)$ respectively.
\end{defn}

\begin{rem}  \label{rem-ext-bldg-axioms}
One shows that, for any $m>0$, any two level $m$ additive norms lie in the same extended apartment (this is in fact, one of the defining axioms of an abstract building). In particular, $\wt{\B}_m(E)$ is a union of extended apartments $\wt{A}_m(L)$ (see \cite[Proposition 1.21]{RTW}).
\end{rem}

\begin{defn}
An $\O$-lattice in $E$ is a full rank $\O$-submodule $\Lambda \subset E$.
\end{defn}
The collection of all $\O$-lattices in $E$ is usually called the \emph{affine Grassmannian} of $\GL(E)$ and plays an important role in representation theory. To every $\O$-lattice $\Lambda$ there corresponds a $\Z$-valued additive norm $\w_\Lambda: E \to \overline{\Z}$ defined by $\w_\Lambda(e) = \max\{j \mid e \in \varpi^j \Lambda\}$. Conversely, if $\w$ is an integer valued additive norm, $\Lambda_\w = E_{\w \geq 0} = \{e \in E \mid \w(e) \geq 0\}$ is an $\O$-lattice. One verifies that $\Lambda \mapsto \w_\Lambda$ and $\w \mapsto \Lambda_\w$ give a one-to-one correspondence between the set of $\O$-lattices in $E$ and the set of integer valued additive norms. We think of integer valued additive norms as lattice points in $\wt{\B}_\aff(E)$.

\begin{defn}
Let $B = \{b_1, \ldots, b_r\}$ be a basis for $E$ spanning a frame $L$. One sees that for an $\O$-lattice $\Lambda$, the corresponding additive norm $\w_\Lambda$ is adapted to $L$ if and only if $$\Lambda = \sum_i \varpi^{a_i} b_i,$$ for some $a_i \in \Z$. We then say that $\Lambda$ is \emph{adapted} to the frame $L$ (or the basis $B$). 
\end{defn}

\begin{rem}[Geometric realizations of Tits and Bruhat--Tits buildings]  \label{rem-realization-space-bldg}
One can give natural constructions of the geometric realization of the Tits building and Bruhat--Tits building of $\GL(E)$ from the spaces $\wt{\B}_\sph(E)$ and $\wt{\B}_\aff(E)$ as follows:
\begin{itemize}
    \item[(i)] Let us say that two valuations $\w$ and $\w'$ on $E$ are equivalent, written $\w \sim \w'$, if there exists $m>0$ and $c \in \R$ such that $\w' = m\w + c$. Let $\mathcal{C}$ denote the set of all constant valuations, that is, valuations that have the same value on all nonzero vectors. Then one can show that the quotient space $\B_\sph(E) = (\wt{\B}_\sph(E) \setminus  \mathcal{C}) / \sim$ can be identified with the geometric realization of the Tits building of $\GL(E)$. The apartments in $\B_\sph(E)$ are obtained from the corresponding quotients of extended apartments in $\wt{\B}_\sph(E)$.
    \item[(ii)] Let us say that two additive norms $\w$ and $\w'$ on $E$ are equivalent, written $\w \sim \w'$, if there exists $c \in \R$ such that $\w' = \w + c$. Then one can show that the quotient space $\B_\aff(E) = \wt{\B}_\aff(E) / \sim$ can be identified with the geometric realization of the Bruhat--Tits building of $\GL(E)$. The apartments in $\B_\aff(E)$ are obtained from the quotients of extended apartments in $\wt{\B}_\aff(E)$.
\end{itemize}
\end{rem}

\subsection{Link of a vertex}
Let $\Lambda \subset E$ be an $\O$-lattice. We can form the $\k$-vector space $$E_\Lambda = \Lambda \otimes_\O \k \cong \k^r.$$
In this section we see that a small neighborhood of $\Lambda$ in $\wt{\B}_\aff(E)$ can be identified with a small neighborhood of the origin in the $\wt{\B}_\sph(E_\Lambda)$.

First, we note that, for any $m \geq 0$, $\wt{\B}_m(E)$ can be equipped with a natural metric topology. We recall that for any basis $B = \{b_1, \ldots, b_r\}$, an apartment $\tilde{A}_m(B)$ can be identified with $\R^r$ via $\w \mapsto \{\w(b_1), \ldots, \w(b_r)\}$. This identification gives a metric on the apartment (Euclidean metric) which does not change if we rescale the basis elements. In other words, the metric does not depend on the choice of the basis for the apartment. Moreover, the metrics glue together to give a metric on the whole $\wt{\B}_m(E)$. In this metric, an open ball of radius $\epsilon > 0$ centered at $\w$ is the collection of all $\w' \in \wt{\B}_m(E)$ whose distance to $\w$ in some common apartment is less than $\epsilon$, that is, the union of open balls centered at $\w$ of radius $\epsilon$ in all the extended apartments $\wt{A}_m(B)$ containing $\w$.

It is well-known in the building theory that the \emph{link} of a vertex in an affine building (regarded as a simplicial complex) is a spherical building. The following proposition is a manifestation of this fact in the case of the Bruhat--Tits building of $\GL(E)$.
\begin{prop} \label{prop-link-vertex}
\begin{itemize}
    \item[(a)]
A neighborhood $U_\Lambda$ of $\Lambda$ in the extended Bruhat--Tits building $\wt{\B}_\aff(E)$ can naturally be identified with a neighborhood $U_\sph$ of the origin in the extended Tits building $\wt{\B}_\sph(E_\Lambda)$.
\item[(b)]  There is a one-to-one correspondence between the extended apartments in $\tilde{\B}_\aff(E)$ that contain $\Lambda$ and the extended apartments in $\tilde{\B}_\sph(E_\Lambda)$, and under the identification of $U_\Lambda$ and $U_\sph$, apartments containing $\Lambda$ go to apartments. 
\end{itemize}
\end{prop}
Proposition \ref{prop-link-vertex} is a consequence of the following two lemmas.
\begin{lem}  \label{lem-lattive-vs-subspace}
There is a one-to-one correspondence between the $\O$-lattices $\Lambda'$ such that $\varpi \Lambda \subset \Lambda' \subset \Lambda$ and the $\k$-vector subspaces in $E_\Lambda$. 
\end{lem}
\begin{proof}
Let $\pi: \Lambda \to E_\Lambda = \Lambda / \varpi \Lambda$ be the quotient map. Then $\pi(\Lambda')$ is a $\k$-subspace in $E_\Lambda$. Conversely, for a subspace $W \subset E_\Lambda$, $\pi^{-1}(W)$ is an $\O$-lattice and $\varpi \Lambda \subset \pi^{-1}(W) \subset \Lambda$.  
\end{proof}
\begin{lem}  \label{lem-w-w-Lambda}
Let $\w_\Lambda$ be the additive norm associated to an $\O$-lattice $\Lambda$ and let $\w$ be another additive norm. Suppose both $\w_\Lambda$ and $\w$ are adapted to a basis $B=\{b_1, \ldots, b_r\}$ and $\w_\Lambda(b_i) = 0$, for all $i$, or equivalently, $\Lambda = \sum_i \O b_i$. Then the following are equivalent: 
\begin{itemize}
\item[(i)] For all $0 \leq a < 1$, $\varpi \Lambda \subset E_{\w \geq a} \subset \Lambda$.
\item[(ii)] For all $i$, $0 \leq \w(b_i) < 1$.
\end{itemize}
\end{lem}
\begin{proof}
Suppose (i) holds. Thus $\w(\sum_i \lambda_i b_i) \geq a$ implies $\w_\Lambda(\sum_i \lambda_i b_i) \geq 0$. This means that if, for all $i$, $\val(\lambda_i)+\w(b_i) \geq a$ then, for all $i$, $\val(\lambda_i) \geq 0$. From this we conclude that, for all $i$, $\w(b_i) < 1$. Similarly, we know that $\w_\Lambda(\sum_i \lambda_i b_i) \geq 1$ implies $\w(\sum_i \lambda_i b_i) \geq a$, for all $0 \leq a < 1$. This then shows that $\w(b_i) \geq 0$, for all $i$. Thus (ii) holds. The converse can be verified in a similar fashion.   \end{proof}

\begin{proof}[Proof of Proposition \ref{prop-link-vertex}]

(a) We recall that an additive norm $\w: E \to \overline{\R}$ is determined by the decreasing $\R$-filtration of $E$ by $\O$-lattices $$E_{\w \geq a} = \{ e \in E \mid \w(e) \geq a\}.$$
Moreover, for any $a \in \R$ we have $E_{\w \geq a+1} = \varpi E_{\w \geq a}$. 
Now consider the set of all additive norms $\w$ such that for any $0 \leq a < 1$ we have $\varpi \Lambda \subset E_{\w \geq a} \subset \Lambda$. Lemma \ref{lem-w-w-Lambda} shows that this is an open neighborhood $U_\Lambda$ of $\w_\Lambda$ in the extended building $\wt{\B}_\aff(E)$. Now given $\w \in U_\Lambda$, define a decreasing $\R$-filtration $(E_{\Lambda, a})_{a \in \R}$ of subspaces in the $\k$-vector space $E_\Lambda$ by: 
$$
E_{\Lambda, a} = 
\begin{cases}
E_\Lambda = \Lambda / \varpi \Lambda & \text{ if } a < 0 \\ 
E_{\w \geq a} / \varpi \Lambda & \text{ if } 0 \leq a < 1 \\
\{0\} & \text{ if } 1 \geq a. \\
\end{cases} 
$$
This decreasing $\R$-filtration determines a valuation $\overline{\w}$ on $E_\Lambda$ given by: 
\[
\overline{\w}(\overline{e}) = \max\left\{a \in \R \; | \; \overline{e} \in E_{\Lambda,a}\right\}.
\]
Conversely, let $U_\sph$ denote the collection of all valuations on $E_\Lambda$ whose non-infinity values lie in $[0, 1)$. Let $\overline{\w}: E_\Lambda \to \overline{\R}$ be a valuation in $U_\sph$. Then $\overline{\w}$ gives rise to a decreasing filtration of $E_\Lambda$ by subspaces $E_{\Lambda, a} := E_{\Lambda, \overline{\w} \geq a}$, for $0 \leq a <1$. This in turn gives a decreasing filtration of $E$ by $\O$-lattices $E_a := \pi^{-1}(E_{\Lambda, a})$, where $\pi: \Lambda \to E_\Lambda$ is the quotient map. We can define $E_{a}$ for all $a \in \R$ by letting $E_{\Lambda, a} = \varpi^k E_{\Lambda, a'}$ where $a = k + a'$ with $k \in \Z$ and $0 \leq a' < 1$. Now let $\w$ be the additive norm on $E$ defined by the decreasing $\R$-filtration by $\O$-lattices $E_{a}$.
One verifies that $\w \mapsto \overline{\w}$ and $\overline{\w} \mapsto \w$ give a one-to-one correspondence between the additive norms in $U_\Lambda$ and the valuations in $U_\sph$.\\
(b) First we note that the apartments in $\tilde{\B}_\aff(E)$ that contain $\Lambda$ are in one-to-one correspondence with the $\O$-bases for $\Lambda$ up to multiplication by units in $\O$.
Let $B = \{b_1, \ldots, b_r\}$ be an $\O$-basis for $\Lambda$. Then $\overline{B}$, the image of $B$ in $E_\Lambda$, is a $\k$-basis for $E_\Lambda$. One then sees that $B \mapsto \overline{B}$ gives a one-to-one correspondence between the $\O$-bases for $\Lambda$, up to multiplication by units in $\O$, and the $\k$-bases for $E_\Lambda$, up to multiplication by nonzero scalars in $\k$. In other words, there is a one-to-one correspondence between the apartments in $\tilde{\B}_\aff(E)$ that contain $\Lambda$ and the apartments in $\tilde{\B}_\sph(E_\Lambda)$. Now, from the constructions of $\w$ and $\overline{\w}$, it follows that $\w \in U_\Lambda$ is adapted to an $\O$-basis $B = \{b_1, \ldots, b_r\}$ for $\Lambda$ if and only if $\overline{\w}$ is adapted to $\overline{B}$. Moreover, we have:
\begin{equation}  \label{equ-w-vs-bar-w}
\w(b_i) = \overline{\w}(\overline{b}_i),~i=1, \ldots, r.    
\end{equation}
This finishes the proof.
\end{proof}

\subsection{Piecewise linear and affine maps} \label{subsec-PL-PA}
In this section, we review some of the main definitions needed for stating the classification result of toric vector bundles over a DVR from \cite{KMT}. 

We start with the case that $\Sigma_0$ is a fan in $N_\R = N_\R \times \{0\}$.
\begin{defn}[Piecewise linear map to the extended Tits building]
We say that a map $\Phi_0: |\Sigma_0| \to \wt{\B}_\sph(E)$ is a \emph{piecewise linear map} if the following hold:
\begin{itemize}
\item[(a)] For each cone $\sigma \in \Sigma_0$, there is a basis $B_\sigma \subset E$ (not necessarily unique) such that $\Phi_0(\sigma)$ lies in the extended apartment $\wt{A}_\sph(B_\sigma)$.
\item[(b)] ${\Phi_0}_{|\sigma}$ is the restriction of a linear map from the linear span of $\sigma$ to $\wt{A}_\sph(B_\sigma)$ \end{itemize}
We say that $\Phi_0$ is an \emph{integral} piecewise linear map if for each $\sigma \in \Sigma$, ${\Phi_0}_{|\sigma}$ is the restriction of an integral linear map from $N_\R$ to $\wt{A}_\sph(\sigma)$. Here integral means it send lattice points to lattice points.  
\end{defn}

\begin{defn}[Morphism of piecewise linear maps]  \label{def-morphism-PL}
Let $E$ and $E'$ be finite dimensional vector spaces. Let $\Phi_0: \Sigma_0 \to \wt{\B}_\sph(E)$, $\Phi'_0: \Sigma_0 \to \wt{\B}_\sph(E')$ be piecewise linear maps. A \emph{morphism} from $\Phi_0$ to $\Phi'_0$ is a linear map $F: E \to E'$ such that for all $x \in \Sigma_0$ and $e \in E$, we have $\Phi_0(x)(e) \leq \Phi'_0(x)(F(e))$.
\end{defn}

Next, let $\Sigma$ be a fan in $N_\R \times \R_{\geq 0}$. Recall that $\Sigma_1$ denotes the intersection of $\Sigma$ with $N_\R \times \{1\}$. It is a polyhedral complex in $\N_\R \times \{1\}$.

\begin{defn}[Graded piecewise linear map to the total extended building]
We say that a map $\Phi: |\Sigma| \to \wt{\B}(E)$ is a \emph{graded piecewise linear map} if the following hold:
\begin{itemize}
\item[(a)] For any $m \geq 0$ and $(x, m) \in |\Sigma| \subset N_\R \times \R$ we have $\Phi(x, m) \subset \wt{\B}_m(E)$. That is, $\Phi$ sends level $m$ points to level $m$ points.
\item[(b)] For each cone $\sigma \in \Sigma$, there is a basis $B_\sigma \subset E$ (not necessarily unique) such that $\Phi(\sigma)$ lies in the extended apartment $\wt{A}(B_\sigma)$.
\item[(c)] $\Phi_{|\sigma}$ is the restriction of a linear map from the linear span of $\sigma$ to $\wt{A}(B_\sigma)$  
\end{itemize}
We say that $\Phi$ is an \emph{integral} graded piecewise linear map if for each $\sigma \in \Sigma$, $\Phi_{|\sigma}$ is the restriction of an integral linear map from $N_\R \times \R_{\geq 0}$ to $\wt{A}(\sigma)$. 
\end{defn}

The notion of morphism of graded linear maps is defined in the same fashion as in Definition \ref{def-morphism-PL}.


Let $\Phi: |\Sigma| \to \wt{\B}(E)$ be a piecewise linear map. Then 
the restriction $\Phi_1: |\Sigma_1| \to \wt{\B}_\aff(E)$ of $\Phi$ to $|\Sigma_1| = |\Sigma| \cap (N_\R \times \{1\})$ is a piecewise affine map in the following sense:
\begin{defn}[Piecewise affine map to $\wt{\B}_\aff(E)$]
Let $\Sigma_1$ be a polyhedral complex in the affine space $N_\R \times \{1\}$.
A map $\Phi_1: |\Sigma_1| \to \wt{\B}_\aff(E)$ is a \emph{piecewise affine} map if the following holds:
\begin{itemize}
\item[(a)] For any polyhedron $\Delta \in \Sigma_1$, there is an extended apartment $\wt{A}_\aff(\Delta)$ (not necessarily unique) such that the restriction ${\Phi_1}_{|\Delta}$ maps $\Delta$ into $\wt{A}_\aff(\Delta)$.
\item[(b)] For any $\Delta \in \Sigma_1$, ${\Phi_1}_{|\Delta}$ is the restriction of an affine map from $N_\R \times \{1\}$ to the affine space $\wt{A}_\aff(\Delta)$.
\end{itemize}
We say that $\Phi$ is an \emph{integral} piecewise affine map if
for each $\Delta \in \Sigma_1$, ${\Phi_1}_{|\Delta}$ is the restriction of an integral affine map from $N_\R \times \{1\}$ to $\wt{A}_\aff(\Delta)$.
\end{defn}

Let $\Phi_1: |\Sigma_1| \to \wt{\B}(E)$ be an integral piecewise affine map. For any polyhedron $\Delta \in \Sigma_1$, the requirement that $\Phi_{1, \Delta} := {\Phi_1}_{|\Delta}: \Delta \to \wt{A}_\aff(\Delta)$ is an integral affine map means that there exists a $K$-basis $B_\Delta = \{b_{\Delta, 1}, \ldots, b_{\Delta, r}\}$ for $E$ and $u(\Delta) = \{u_{\Delta, 1}, \ldots, u_{\Delta, r}\} \subset M$ such that for any $x \in \sigma$, $\Phi_{1, \Delta}(x)$, as an additive norm on $E$, is given by:
\begin{equation}   \label{equ-Phi-sigma} 
\Phi_{1, \Delta}(x)\left(\sum_i \lambda_i b_{\Delta, i}\right) = \min\{ \val(\lambda_i) + \langle u_{\Delta, i}, x \rangle \mid i=1, \ldots, r\}.
\end{equation}

\section{Classification of toric vector bundles over a field and over a DVR}\label{sec:class}
In this section, we give the statement of classification of toric vector bundles over a DVR in terms of graded piecewise linear maps to the total building $\wt{\B}(E)$. 

First we recall the case of toric vector bundles over a field and how to recover their equivariant Chern classes. The following is a reformulation of Klyachko's classification of toric vector bundles over a field (see \cite[Section 2]{KM-valuation}).
\begin{thm}[Classification of toric vector bundles over a field]
To each toric vector bundle $\E$ over a toric variety $X_\Omega$ over a field, corresponding to a fan $\Omega$, there corresponds a piecewise linear map $\Phi_\E: |\Omega| \to \wt{\B}_\sph(E)$. Moreover, this gives an equivalence of categories between toric vector bundles and piecewise linear maps.
\end{thm}

The equivariant Chern classes of a toric vector bundle can be immediately read off from its corresponding piecewise linear map (\cite[Prop. 3.1]{Payne-chern} or \cite[Corollary 3.5]{KM}). Recall that the $i$-th elementary symmetric function $\epsilon_i$ on $\R^r$ naturally induces a function $\epsilon_i: \wt{\B}_\sph(E) \to \R$ as follows: let $  \w: E \to \overline{\R}$ be a valuation with value set $\w(E) = \{a_1, \ldots, a_k, \infty\}$. We define $\epsilon_i(\w)$ to be the value of the $i$-th elementary symmetric function on $a_1, \ldots, a_k$ where each $a_j$ is 
repeated $\dim(E_{\w \geq a_j} / E_{\w > a_j})$ times. 

Consider the equivariant Chern class $c_i^T(\E) \in \CH^i_T(X_{\Omega})$ of a  toric vector bundle $\E$ over $X_{\Omega}$. We have already seen in Theorem~\ref{th:payne-chow} that it can be represented by a piecewise polynomial function of degree $i$ with respect to the fan $\Omega$.

\begin{thm}[Equivariant Chern classes]   \label{th-equiv-Chern-tvb-field}
The $i$-th equivariant Chern class of $\E$ is represented by the piecewise polynomial function $\epsilon_i \circ \Phi_\E: |\Omega| \to \R$.    
\end{thm}

Finally, we recall the classification of toric vector bundles over a DVR from \cite{KMT}. 
\begin{thm}[Classification of toric vector bundles over a DVR]  \label{th:vb-dvr}
To each toric vector bundle $\E$ over a toric scheme $\X_\Sigma$ there corresponds a graded piecewise linear map $\Phi_\E$ from $|\Sigma|$ to the total extended building $\wt{\B}(E)$. If $\rec(\Sigma_1) = \Sigma_0$ is a fan, then equivalently, to $\E$ we can associate the piecewise affine map $\Phi_{\E, 1} = {\Phi_{\E|}}_{|\Sigma_1|}: |\Sigma_1| \to \wt{\B}_1(E) = \wt{\B}_\aff(E)$. Moreover, this gives rise to an equivalence of categories. Furthermore, the restriction $\Phi_{\E, 0} = {\Phi_{\E|}}_{|\Sigma_0|} \colon |\Sigma_0| \to \wt{\B}_0(E) = \wt{\B}_{\sph}(E)$ is the piecewise linear map corresponding to the restriction of $\E$ to the generic fiber. 
\end{thm}

 
\section{Toric vector bundles over irreducible components of the special fiber}\label{sec:special-fiber}
Let $\Sigma$ be a fan in $N_{\R} \times \R_{\geq 0}$. Recall that $\Sigma_1$ and $\Sigma_0$ denote the intersections with $N_{\R} \times \{1\}$ and with $N_{\R} \times \{1\}$, respectively. Assume that $\rec(\Sigma_1) = \Sigma_0$ is a fan and that the vertices of all $\Delta \in \Sigma_1$ are in the lattice $N \times \{1\}$.

Let $\E$ be a toric vector bundle over the toric scheme $\X = \X_\Sigma$ with corresponding piecewise affine map $\Phi_1: |\Sigma_1| \to \wt{\B}_\aff(E)$. Let $\nu$ be a vertex of the polyhedral complex $\Sigma_1$ and let $\X_{s,\nu}$ denote the irreducible component of the special fiber $\X_s$ corresponding to $\nu$. This is the same as the vertical cycle $V(\nu)$ from Definition \ref{def:vertical-cycles} and it is itself a toric variety for the action of the residue torus $T_{s}$ corresponding to the star fan $\Sigma_1(\nu)$. Thus, $\E_{|\X_{s,\nu}}$ is a toric vector bundle over the toric variety $\X_{s,\nu}$ and hence corresponds to a piecewise linear map $\Phi_\nu$ from the fan of $\X_{s,\nu}$ to an extended Tits building. In this section we give a natural description of this piecewise linear map as $\Phi_1$ restricted to a small neighborhood of $\nu$ which then lands in a small neighborhood of $\Lambda = \Phi_1(\nu)$. By Proposition \ref{prop-link-vertex}, this small neighborhood can be identified with a small neighborhood of the extended Tits building $\wt{\B}_\sph(E_\Lambda)$, where $E_\Lambda$ denotes the $\k$-vector space $\Lambda / \varpi \Lambda$. Thus, restriction of $\Phi_1$ to a small neighborhood of $\nu$ gives rise to a piecewise linear map $\Phi_\nu$ from the star $\Sigma_1(\nu)$, of the polyhedral complex $\Sigma_1$ at $\nu$, to $\wt{\B}_\sph(E_\Lambda)$.

First we need the following result which shows that $\Phi_\nu: |\Sigma_1(\nu)| \to \tilde{\B}_\sph(E_\Lambda)$ is actually a piecewise linear map and describes its piecewise linear data in terms of piecewise affine data of $\Phi_1$.
\begin{prop}   \label{prop-Phi-nu-PL}
With notation as above, let $\Delta$ be a polyhedron in $\Sigma_1$ with vertex $\nu$. Let $\Delta_\nu$ be the corresponding cone in the star fan $\Sigma_1(\nu)$. 
Let $B_\Delta = \{b_{\Delta,1}, \dotsc, b_{\Delta,r}\} \subset E$ be the $\O$-basis for $\Lambda$ with multiset of characters $u(\Delta) = \{u_{\Delta,1}, \dotsc, u_{\Delta,r}\}$ giving the affine data of ${\Phi}_{1, \Delta}$ (see \eqref{equ-Phi-sigma}). That is, for all $x \in \Delta$, $\Phi_1(x)$ is the additive norm on $E$ adapted to $B_\Delta$ with: $$\Phi_1(x)(b_{\Delta, i}) = \langle u_{\Delta, i}, x \rangle, ~i=1, \ldots, r.$$
Then the piecewise linearity data of $\Phi_\nu$ on the cone $\Delta_\nu$ is given by the basis $\overline{B}_\Delta = \{\overline{b}_{\Delta,1}, \dotsc, \overline{b}_{\Delta,r}\}$, the image of $B_\Delta$ in $E_\Lambda$, and the same multiset of characters $u(\Delta)$. That is, for $x \in \Delta_\nu$, $\Phi_\nu(x)$ is the valuation on $E_\Lambda$ adapted to $\overline{B}_\Delta$ with:
$$\Phi_\nu(x)(\overline{b}_{\Delta, i}) = \langle u_{\Delta, i}, x \rangle, ~i=1, \ldots, r.$$
\end{prop}
\begin{proof}
The claim follows from Proposition \ref{prop-link-vertex}(2) and in particular Equation \eqref{equ-w-vs-bar-w}.
\end{proof}

Recall that for each cone $\sigma \in \Sigma$ we have an affine toric scheme $\U_\sigma \subset \X_\Sigma$ defined as the spectrum of the ring:
$$R_\sigma = \O[\chi^u \varpi^k \mid (u, k) \in \sigma^\vee \cap \wt{M}].$$
When $\sigma$ is the cone over a polyhedron $\Delta \in \Sigma_1$, we denote the scheme $\U_\sigma$ and the ring $R_\sigma$ also by $\U_\Delta$ and $R_\Delta$ respectively.
In particular, the open orbit $\U_0 \cong T_{\eta}$ is the spectrum of the Laurent polynomial ring $R_0 = K[\chi^u \mid u \in M]$. The homomorphism that sends every $\chi^u$ to $1$ defines a $K$-point $x_0 \in \U_0$. 
A polyhedron $\Delta \in \Sigma_1$ that has $\nu$ as a vertex determines an affine chart $\overline{\U}_{\nu, \Delta}$ in $\X_{s, \nu}$ defined as the spectrum of the ring:
\begin{equation*}
\overline{R}_{\nu, \Delta} = \k[\chi^u \mid (u, -\langle u, \nu \rangle) \in \sigma^\vee \cap \wt{M} ], \end{equation*}
where $\sigma \in \Sigma$ is the cone over the polyhedron $\Delta$.
When $\Delta = \nu$, the ring $\overline{R}_{\nu, \nu} = \k[M]$ is the coordinate ring of the open $T_s$-orbit $\overline{\U}_{\nu, \nu}$ in $\X_{s, \nu}$. The homomorphism that sends every $\chi^u$ to $1$, defines a $\k$-point $x_\nu \in \overline{\U}_{\nu, \nu} \subset \X_{s, \nu}$. We take $x_{\nu}$ to be the distinguished point (in the open $T_s$-orbit) in $\X_{s,\nu}$.

Consider the affine scheme $\U_\nu$. By definition of the $\O$-lattice $\Lambda = \Phi_1(\nu) \subset E$ (see \cite[Equation (13)]{KMT}), the space of sections $H^0(\U_\nu, \E_{|\U_\nu}) \subset K[M] \otimes_K E$ can be described as:
$$H^0(\U_\nu, \E_{|\U_\nu}) = \sum_{u \in M} \chi^u \varpi^{-\langle u, \nu \rangle} \Lambda.$$
We have a homomorphism $R_\nu \to \O$ given by $\chi^u \varpi^{-\langle u, \nu \rangle} \mapsto 1$, $\forall u \in M$. It corresponds to the morphism $S=\Spec(\O) \to \U_\nu$ which we think of as the infinitesimal curve in $\U_\nu$, corresponding to $\nu$, approaching the point $x_\nu$ in the special fiber.
Then the fiber $\E_{x_\nu}$ is obtained as:
$$\E_{x_\nu} = H^0(\U_\nu, \E_{|\U_\nu}) \otimes_{R_\nu} \k,$$
where we consider $\k$ as an $R_\nu$-module via the homomorphism $R_\nu \to \O \to \k$.  Thus, one sees that $$\E_{x_\nu} = (\sum_{u \in M} \chi^u \varpi^{-\langle u, \nu \rangle} \Lambda) \otimes_{R_\nu} \k = (\sum_{u \in M} \Lambda) \otimes_\O \k = \Lambda \otimes_\O \k.$$ That is, the fiber $\E_{x_\nu}$ can naturally be identified with $E_\Lambda = \Lambda \otimes_\O \k$. 

\begin{thm}[Piecewise linear map corresponding to restriction of $\E$ to $\X_{s,\nu}$]  \label{th:link}
Consider the piecewise linear map $\Phi_\nu: |\Sigma_1(\nu)| \to \wt{\B}_\sph(E_\Lambda)$ obtained by restricting the piecewise affine map $\Phi_1$ to a small neighborhood of $\nu$ in $\Sigma_1(\nu)$ and by looking at a small neighborhood (link) of $\Lambda = \Phi_1(\nu)$ in $\wt{\B}_\aff(E)$ which we can naturally identify with a small neighborhood of the origin in $\wt{\B}_{\sph}(E_\Lambda)$ (Proposition \ref{prop-link-vertex}). Then $\Phi_\nu$ is the piecewise linear map corresponding to the restriction of $\E$ to the $T_{s}$-toric variety $\X_{s,\nu}$.
\end{thm}
\begin{proof}
We would like to show that $\Phi_\nu$ is the piecewise linear map associated to the $T_{s}$-toric vector bundle $\E_{|\X_{s,\nu}}$. To show this, we verify that the local equivariant triviality data of $\E_{|\X_{s,\nu}}$ agrees with the piecewise linearity data of $\Phi_\nu$. As above, let $\Delta$ be a polyhedron in $\Sigma_1$ with vertex $\nu$, 
and $\U_\Delta \subset \X_\Sigma$ the affine toric scheme corresponding to $\Delta$ and $\overline{\U}_{\nu, \Delta}$ the affine toric chart corresponding to $(\nu, \Delta)$ in the $T_{s}$-toric variety $\X_{s, \nu}$. {We recall that $\overline{\U}_{\nu, \Delta}$ is the affine toric subvariety of $\X_{s,\nu}$ corresponding to $\Delta_v \in \Sigma_1(\nu)$, the cone generated by $\Delta$ in $\Sigma_1(\nu)$}. 
We have an inclusion morphism $\iota: \overline{\U}_{\nu, \Delta} \to \U_{\Delta}$ 
which is the restriction of the inclusion morphism $\X_s \to \X$. It corresponds to the ring homomorphism $$\iota^*: R_\Delta=\O[\chi^u \varpi^k \mid (u, k) \in \sigma^\vee \cap \wt{M}] \to \overline{R}_{\nu, \Delta} = \k[\chi^u \mid (u, -\langle u, \nu \rangle) \in \sigma^\vee \cap \wt{M}],$$ defined by $\O \to \k$ and: 
$$\iota^*(\chi^u \varpi^k) = \begin{cases}
\chi^u~ \text{ if } k = -\langle u, \nu \rangle\\
0 ~\text{ if } k > -\langle u, \nu \rangle.
\end{cases}$$
Let $B_\Delta = \{b_{\Delta, 1}, \ldots, b_{\Delta, r} \} \subset E$ and $u(\Delta) = \{u_{\Delta, 1}, \ldots, u_{\Delta, r}\} \subset M$ be the $\O$-basis for $\Lambda = \Phi_1(\nu)$ and the multiset of characters corresponding to $\Delta$ respectively. Thus, $\Lambda = \bigoplus_{i=1}^r \O b_{\Delta, i} \subset E$ and we have:
$$H^0(\U_\Delta, \E_{|\U_\Delta}) = R_\Delta \otimes_\O \Lambda = \bigoplus_{i=1}^r R_\Delta b_{\Delta, i} \subset K[M] \otimes_K E.$$
We note that:
$$H^0(\overline{\U}_{\nu, \Delta}, \E_{|\overline{\U}_{\nu, \Delta}}) = H^0(\U_\Delta, \E_{|\U_\Delta}) \otimes_{R_\Delta} \overline{R}_{\nu, \Delta},$$ where we regard $\overline{R}_{\nu, \Delta}$ as an $R_\Delta$-module via the homomorphism $\iota^*$ above. We thus conclude that:
\begin{align*}
H^0(\overline{\U}_{\nu, \Delta}, \E_{|\overline{\U}_{\nu, \Delta}}) &= (R_\Delta \otimes_\O \Lambda) \otimes_{R_\Delta} \overline{R}_{\nu, \Delta} = (\bigoplus_{i=1}^r R_\Delta b_{\Delta, i}) \otimes_{R_\Delta} \overline{R}_{\nu, \Delta}\\
&= \overline{R}_{\nu, \Delta} \otimes_\k E_\Lambda = \bigoplus_{i=1}^r \overline{R}_{\nu, \Delta} \overline{b}_{\Delta, i} \subset \k[M] \otimes_\k E_\Lambda,\\
\end{align*}
where $\overline{b}_{\Delta, i}$ denotes the image of $b_{\Delta, i}$ in $E_\Lambda = \Lambda  \otimes_\O \k = \Lambda / \m \Lambda$.
The claim now follows from Propositions \ref{prop-link-vertex} and \ref{prop-Phi-nu-PL}.
\end{proof}

\section{Equivariant Chern classes of toric vector bundles over a DVR }

Given a $T_s$-scheme $X$ and and an equivariant vector bundle $E$, following \cite[Section 2.4]{EG} one defines equivariant Chern classes $c_i^{T_{s}}(E) \in \CH^i_{T_{s}}(X)$ in the equivariant Chow cohomology groups of $X$. 

Now, let $\Sigma, \Sigma_1$ and $\Sigma_0$ be as in Section \ref{sec:equiv-chow}. In particular, $\Sigma_1$ is complete and regular, and its vertices are contained in the lattice $N\times\{1\}$.
\begin{defn}\label{def:eq-chern} Given a toric vector bundle $\E$ over the proper, regular toric scheme $\X_{\Sigma_1}$ over $S$, one defines $c_i^T(\E)$, the $i$-th equivariant Chern class of $\E$, by its restriction to the special fiber, i.~e.
\[
c_i^T(\E) \coloneqq c_i^T(\iota^*\E) \in \CH^i_{T_s}(\X_s),
\]
where $\iota \colon \X_s \hookrightarrow \X_{\Sigma}$ denotes the canonical inclusion morphism. By Theorem \ref{th:cohomology-special} this is a piecewise polynomial function on $\Sigma_1$ of degree $i$.
\end{defn}

Recall that for an $\O$-lattice $\Lambda$, we denote the $\k$-vector space $\Lambda \otimes_\O \k$ by $E_\Lambda$. To this vector space there corresponds the extended Tits building $\wt{\B}_\sph(E_\Lambda)$.
Also we recall that for each $i=1, \ldots, r$ we have the $i$-th elementary symmetric function $\epsilon_i: \wt{\B}_\sph(E_\Lambda) \to \R$ (see Section \ref{sec:class}). 

\begin{thm}[The piecewise polynomial map representing the $i$-th equivariant Chern class] \label{th:equi}
Let $\E$ be a toric vector bundle on a toric scheme $\X_\Sigma$. Let $c_i^T$ denote the $i$-th equivariant Chern class of $\E$ represented by a piecewise polynomial function $(f_\nu)$ where $\nu$ runs over the vertices in the polyhedral complex $\Sigma_1$. Then the piecewise polynomial function $f_\nu$ on the fan $\Sigma_1(\nu)$, the star of $\Sigma_1$ at $\nu$, is given by $\epsilon_i \circ \Phi_\nu$.  \end{thm}
\begin{proof}
The claim is an immediate corollary of Theorem \ref{th:link} and Theorem \ref{th-equiv-Chern-tvb-field}. 
\end{proof}
\begin{rem} The above classification might be useful in order to consider moduli spaces of toric vector bundles over a DVR with fixed equivariant total Chern class.
\end{rem}
We can also recover the equivariant Chern classes over the generic fiber. Let $\E$ be a toric vector bundle on a toric scheme $\X_\Sigma$. Let $\Phi_{\E} \colon |\Sigma| \to \wt{\B}(E)$ be the corresponding graded piecewise linear map to the total extended building. The restriction to $|\Sigma_0|$ gives rise to a piecewise linear map $\Phi_{\E,0} \colon |\Sigma_0| \to \wt{\B}_0(E)=\wt{\B}_\sph(E)$ to the spherical Tits building. This is the piecewise linear map corresponding to the restriction of $\E$ to the generic fiber $\X_{\Sigma, \eta}$ (\cite{KMT}). As a corollary of Theorem \ref{th-equiv-Chern-tvb-field} we obtain the following.
\begin{cor}\label{cor:recovering}
The equivariant Chern classes of the restriction of $\E$ to the generic fiber are represented by the piecewise polynomial functions 
\[
\epsilon_i \circ \Phi_{\E,0}.
\]
  \end{cor}

\section{Connections to the Arakelov theory of toric varieties}\label{sec:arakelov}
 The goal of Arakelov theory is to do intersection theory over rings of integers of number fields. The main idea is as follows. Consider an algebraic variety over a number field $K$. Then classical Arakelov theory (\cite{arakelov1}, \cite{arakelov2}, \cite{Fa}, \cite{GS1}, \cite{GS2}) considers algebraic cycles and vector bundles on a fixed regular model $\X$ of $X$ over the ring of integers of $K$, together with analytic data on the complex points of $X$. Adding this data can be thought of as a \emph{compactification} of $\X$ and leads to an arithmetic intersection theory on $X$. 

One of the main advents of Arakelov intersection theory is the concept of the \emph{height} of a subvariety, defined in pure analogy with its degree, replacing classical
intersection theory with arithmetic intersection theory. This has provided many important breakthroughs, starting with Vojta's new proof of the Mordell conjecture \cite{V}, or Faltings' proof of Serge Lang's generalization of the Mordell conjecture \cite{Fa1}.

Around 2010, the so called \emph{toric dictionary} was extended by J.~Burgos, P.~Philippon and M.~Sombra to the Arakelov theory of toric line bundles on toric varieties in \cite{BPS}. For example an arithmetic analogue of the Bernštein-Kušnirenko-Khovanskii theorem is given. 

Let us briefly explain some of their results. We refer to \cite{BPS} for details. Let $K$ be either $\mathbb{C}$ or a discretely valued field. Let $T$ be a split torus over $K$ and $X$ a smooth, complete toric variety over $K$ with torus $T$. Let $\mathcal{L}$ be a toric line bundle on $X$ and let $X^{\operatorname{an}}$ and $\mathcal{L}^{\operatorname{an}}$ be the corresponding analytifications. If $K =\mathbb{C}$ then $X^{\operatorname{an}}$ is just the complex space $X(\mathbb{C})$, whereas in the non-Archimedean case it is the
Berkovich space associated to $X$. 

The basic metrics that one puts on $\mathcal{L}^{\operatorname{an}}$ are the smooth metrics in the Archimedean case and the algebraic metrics in the non-Archimedean case, that is, metrics induced by an integral model of the pair $\left(X,\mathcal{L}\right)$. In the Archimedean case, one has a semipositivity notion in terms of currents, while in the non-Archimedean case, the semipositivity notion comes from algebraic positivity notions of the corresponding divisors (see \cite{BPS} for details). Uniform limits of semipositive smooth, respectively model metrics leads to the notion of semipositive metrics on $\mathcal{L}^{\operatorname{an}}$ in the Archimedean and non-Archimedean cases, respectively.

Let $X$ be a complete toric variety. As usual, for any geometric object on $X$ for which one hopes to have a combinatorial description, one asks for a torus-invariance property. A metric $\| \cdot \| = \left\{\| \cdot \|_x\right\}_{x \in X^{\an}}$ on (the analytifiaction of) a toric line bundle $\ca L^{\on{an}}$ is \emph{toric} if $t \mapsto \|s(t)\|_t$ is invariant under the compact torus $\mathbb{S} \subset T^{\on{an}}$ for any section $s$ of $\ca L$.

The first main arithmetic addition to the toric dictionary is the following classification result (see \cite[Theorem 4.8.1]{BPS}). 
\begin{thm}
    Let $K$ be either $\C$ or a discretely valued field. Then there is a bijection between the space of semipositive toric
metrics on $\ca L^{\on{an}}$ and the space of concave functions $\psi$ on $N_{\R}$ such that $\left|\psi-\gamma_{\ca L}\right|$ is bounded. Here, $\gamma_{\ca L}$ denotes the piecewise linear function on $N_{\R}$ associated to the toric line bundle $\ca L$ on $X$. 
\end{thm}

We expect that the above theorem extends to higher rank toric vector bundles using the theory of buildings. The present article is an indictaion of this expectation in the non-Archimedean case. 

Let $K$ be as above and let $\E$ be a toric vector bundle over the toric variety $X$. A continuous metric on $\E^{an}$ is a continuous family of norms $\| \cdot \|_{\E^{\on{an}}} = \left\{ \| \cdot \|_x\right\}_{x \in X^{\on{an}}}$. The metric is said to be \emph{toric} if $t \mapsto \|s(t)\|_t$ is invariant with respect to the compact torus $\mathbb{S}\subset T^{\on{an}}$ for any (invariant) section~$s$. As in the case of line bundles, if $K$ is discretely valued, then a toric model $\left(\X,\wt{\E}\right)$ of $(X,\E)$ induces a toric metric on $\E^{\on{an}}$. Metrics arising in this way are also called \emph{model metrics}.

A metric on $\E^{an}$ is \emph{semipositive} if the line bundle $\mathcal{O}_{\E}(1)$ on the projectivized total space $\mathbb{P}(\E)$, endowed with the Fubini-Study metric, is semipositive. For model metrics, this is equivalent to the corresponding model being nef. 

The characterization of equivariant Chern classes of toric vector bundles over a DVR given Theorem \ref{th:equi} is a first step towards a combinatorial classification of toric semipositive metrics in the non-Archimedean case. Indeed, the equivariant Chern classes of a model $\wt{\E}$ are the equivariant, non-Archimedean analogues of hermitian Chern forms.

We expect the following classification result. 

\begin{conj}\label{conj:class-metrics}
    Let $K$ be a discretely valued field and let $\E$ be a toric vector bundle over the toric variety $X$ of rank $r$. Let $\Phi_{\E}$ denote the piecewise linear map to the extended spherical Tits building $\wt{\B}_{\sph}(\GL(r,K))$ corresponding to $\E$. Then there is a bijection between the space of semipositive toric metrics on $\E^{\on{an}}$ and the space of ``concave'' piecewise affine functions $N_{\R}  \to \wt{\B}_{\aff}(\GL(r,K))$ satisfying an asymptotic growth condition with respect to $\Phi_{\E}$.
\end{conj}

Note that the the appropriate definition of ``concavity'' for such functions is yet to be determined. 
The specific asymptotic condition is also to be determined. However, note that intuitively this makes sense since the (extended) spherical Tits building can be seen as the boundary of a compactification of the affine building.

A motivation for aiming at a classification as in Conjecture \ref{conj:class-metrics} is to generalize and find non-Archimedean analogues of Chern currents associated to singular semipositive hermitian metrics, by taking projective limits of Chern classes, indexed over all toric models. 

Another natural question is: what is the Archimedean analogue? 
More precisely, it is known that Bruhat--Tits buildings serve as non-Archimedean analogues of symmetric spaces. The latter parametrize toric hermitian metrics on the torus. In analogy to the non-Archimedean case, spherical Tits buildings can be also seen as boundaries of symmetric spaces. Hence, also here it makes sense to think of toric metrics as continuous functions to the symmetric space with an asymptotic behaviour controlled by the spherical building. These considerations are work in progress.

A classification as described above gives an important step towards a complete combinatorial description of the Arakelov theory of toric varieties.
Indeed, a long term goal is a combinatorial description of the arihmetic Groethendieck--Riemann--Roch theorem for toric varieties. Among many other things, we need to study direct images and Bott--Chern classes of toric vector bundles,  and try to describe these in combinatorial terms.  The classification of semipositive metrics on toric vector bundles and the relation to the theory of buildings seems to be a promising step in this direction.

\printbibliography

\end{document}